\documentclass[11pt,a4paper]{amsart}
\usepackage[left=3.5cm,right=3.5cm,top=3.5cm,bottom=3.5cm]{geometry}
\usepackage{enumerate}
\usepackage{leftidx}
\usepackage{mathtools}
\usepackage{amsmath,amscd,amssymb,amsthm}
\usepackage[arrow,matrix]{xy}
\usepackage[OT2,T1]{fontenc}
\usepackage{graphicx,setspace}
\usepackage{hyperref}

\DeclareMathOperator{\gal}{Gal}

\DeclareMathOperator{\im}{Im}

\DeclareMathOperator{\GL}{GL}

\theoremstyle{definition}
\newtheorem{definition}{Definition}[section]

\newtheorem{remark}[definition]{Remark}
\newtheorem*{remark*}{Remark}

\theoremstyle{plain}
\newtheorem{theorem}[definition]{Theorem}
\newtheorem{corollary}[definition]{Corollary}
\newtheorem{lemma}[definition]{Lemma}
\newtheorem{proposition}[definition]{Proposition}
\newtheorem{conjecture}[definition]{Conjecture}

\newcommand{\F}{{ \mathbb F }}
\newcommand{\N}{{ \mathbb N }}
\newcommand{\Q}{{ \mathbb Q }}
\newcommand{\Z}{{ \mathbb Z }}

\newcommand{\p}{{ \mathfrak p }}
\newcommand{\m}{{ \mathfrak m }}
\renewcommand{\O}{{ \mathcal O }}

\author[A. Ferraguti]{Andrea Ferraguti}
\address{ICMAT, Campus de Cantoblanco, 13-15 Calle de Nicol\'as Cabrera, 28049 Madrid, Spain}
\email{and.ferraguti@gmail.com}

\author[C. Pagano]{Carlo Pagano}
\address{Max Planck Institute for Mathematics, Vivatsgasse 7, 53111 Bonn, Germany
}
\email{carlein90@gmail.com}

\title{Constraining images of quadratic arboreal representations}

\begin{document}

\begin{abstract}
In this paper, we prove several results on finitely generated dynamical Galois groups attached to quadratic polynomials. First we show that, over global fields, quadratic post-critically finite polynomials are precisely those having an arboreal representation whose image is topologically finitely generated. To obtain this result, we also prove the quadratic case of Hindes' conjecture on dynamical non-isotriviality. Next, we give two applications of this result. On the one hand, we prove that quadratic polynomials over global fields with abelian dynamical Galois group are necessarily post-critically finite, and we combine our results with local class field theory to classify quadratic pairs over $\Q$ with abelian dynamical Galois group, improving on recent results of Andrews and Petsche. On the other hand we show that several infinite families of subgroups of the automorphism group of the infinite binary tree cannot appear as images of arboreal representations of quadratic polynomials over number fields, yielding unconditional evidence towards Jones' finite index conjecture. 

\end{abstract}

\maketitle

\section{Introduction}

One of the central tasks in modern Arithmetic Dynamics is to understand images of arboreal Galois representations or, more in general, dynamical Galois groups. This is still a rather mysterious topic even in the first non-trivial case, which is that of quadratic polynomials. A great amount of interest on the topic has risen in recent years, as witnessed by the numerous papers on the topic, such as \cite{ahmad,anderson,andrews,boston,ferra2,ferra4,ferra1,ferra3,ferra5,jones_survey,juul}.

In the present paper we prove several results which give constraints on the image of the arboreal representation associated to a quadratic polynomial over a global field. Let us start by briefly recalling the setting: $K$ is a field of characteristic not 2, $f=(x-a)^2-b\in K[x]$ and $\alpha\in K$. Throughout the paper, we will denote by $f^n$ the $n$-fold \emph{iteration} of $f$, rather than the $n$-th power. For every $n\geq 1$, we will denote by $K_n(f,\alpha)$ the splitting field of the polynomial $f^n(x)-\alpha$ and by $G_n(f,\alpha)$ its Galois group. Moreover, we will let $K_\infty(f,\alpha)\coloneqq \varinjlim_n K_n(f,\alpha)$ and $G_\infty(f,\alpha)=\varprojlim_n G_n(f,\alpha)$ the Galois group of $K_\infty(f,\alpha)$. The group $G_\infty(f,\alpha)$ is commonly referred to, in the literature, as the \emph{dynamical Galois group} attached to the pair $(f,\alpha)$. When $f^n-\alpha$ is separable for every $n\geq 1$, one can associate to the pair $(f,\alpha)$ a so-called \emph{arboreal Galois representation}, namely a continuous homomorphism $\rho_{f,\alpha}\colon \gal(K^{\text{sep}}/K)\to \Omega_\infty$, where $\Omega_\infty$ is the group of automorphisms of the infinite, rooted, regular binary tree (see for example \cite{jones_survey} for more details). Clearly, in this case we have that $G_\infty(f,\alpha)\cong \im(\rho_{f,\alpha})$ as topological groups. In order not to confuse the reader, we will use the notation $G_\infty(f,\alpha)$ when there will be no separability assumption, while $\im(\rho_{f,\alpha})$ will be reserved to cases in which we assume every $f^n-\alpha$ to be separable.

Although certain basic results on arboreal representations, such as some of those contained in \cite{ferra5,stoll}, hold for any base field of characteristic not 2, when $K$ is a global field things become much more rigid. In this framework, post-critically finite polynomials (PCF for short) play a crucial role. Recall that a quadratic polynomial $f=(x-a)^2-b$ is called \emph{post-critically finite} if the orbit of $a$ is finite, and \emph{post-critically infinite} otherwise. It is a well known fact that arboreal representations of quadratic PCF polynomials have topologically finitely generated image, which is implied by the fact that in this case the pro-$2$-extension $K_\infty(f,\alpha)$ turns out to be ramified only at finitely many primes. In fact, $f$ is PCF if and only if $K_\infty(f,\alpha)$ ramifies at finitely many primes (see \cite{bridy} for a more general result).

Our first result (Theorem \ref{PCF_equivalence}) shows that for quadratic polynomials over global fields of characteristic not 2 the following four conditions are equivalent:
\begin{enumerate}
\item $f$ is PCF;
\item $K_\infty(f,\alpha)$ ramifies at finitely many primes;
\item $\im(\rho_{f,\alpha})$ is topologically finitely generated;
\item the sub-$\F_2$-vector space of $K^{*}/K^{* 2}$ generated by the set $\{f^n(a)-\alpha\}_{n\geq 1}$ is finite dimensional.
\end{enumerate}
The proof of (3)$\implies$(4) uses part of the results of \cite{ferra5}. In order to prove (4)$\implies$(1) when the base field has positive characteristic, we prove (Theorem \ref{isotriviality}) the quadratic case of Hindes' conjecture on dynamical non-isotriviality \cite[Conjecture 3.1]{hindes1}.

The next two sections of the paper are devoted to apply Theorem \ref{PCF_equivalence} in the context of  Jones' Finite Index Conjecture \cite[Conjecture 3.11]{jones_survey}. This conjecture aims to characterize quadratic rational functions over global fields whose arboreal representation has infinite index image inside $\Omega_\infty$. When one restricts to quadratic polynomials over number fields, the conjecture reads in the following way.
\begin{conjecture} \label{Jones conjecture}
Let $K$ be a number field, and let $f\in K[x]$ be a quadratic polynomial such that $f^n$ is separable for every $n$. Then $\im(\rho_{f,0})$ has infinite index in $\Omega_\infty$ if and only if $f$ is PCF or $0$ is periodic for $f$.
\end{conjecture}
It is worth remarking that one direction of Conjecture \ref{Jones conjecture} is already known to be true: if $f$ is PCF or $0$ is periodic for $f$, then $[\Omega_\infty:\im(\rho_{f,0})]=\infty$ (see \cite[Section 3]{jones_survey}). The converse statement remains unproven, although there are partial results towards it \cite{juul}, conditional on abc and Vojta's conjectures.

In Section \ref{sec:abelian} we turn our attention to the case of \emph{abelian} subgroups of $\Omega_{\infty}$. In \cite[Proposition 7.3]{ferra5} we proved that elements $\sigma$ of $\Omega_{\infty}$ swapping the $2$ halves of the tree do not commute with any other automorphism of the tree having image in $\Omega_{\infty}^{\text{ab}}$ linearly independent from that of $\sigma$. In \cite{ferra5} we used this result, among other tools, to prove \cite[Theorem C]{ferra5}, which gives that any two distinct maximal closed subgroups of $\Omega_\infty$ are not isomorphic as profinite groups. Indeed this result indicates, in a precise sense, that often elements of $\Omega_{\infty}$ do not commute. So, as we show in Proposition \ref{abelian subgroups are of infinite index}, one can immediately apply \cite[Proposition 7.3]{ferra5} to prove that abelian subgroups of $\Omega_{\infty}$ have infinite index therein. Hence each of them is within the realm of Conjecture \ref{Jones conjecture}. Our idea in the present work is to use  \cite[Proposition 7.3]{ferra5} as a tool to systematically detect constraints on abelian arboreal images. Namely we completely prove this case of  Conjecture \ref{Jones conjecture} in the following result. 
\begin{theorem} \label{thm 4}
Conjecture \ref{Jones conjecture} holds whenever $\im(\rho_{f,0})$ is abelian. In particular, if $\im(\rho_{f,0})$ is abelian then $f$ is post-critically finite. Hence, an abelian closed subgroup of $\Omega_{\infty}$ can be an arboreal image over a number field only if it is topologically finitely generated.
\end{theorem}
In fact Theorem \ref{thm 4} holds in the more general setting of dynamical Galois groups over global fields, cf.\ Theorem \ref{abelian_image}.

It is worth noticing that Theorem \ref{thm 4} fits very naturally in the philosophy according to which post-critically finite polynomials are dynamical analogues of elliptic curves with complex multiplication. This way of thinking is supported by several existing results, such as \cite{benedetto2,ghioca}, and of course also by Conjecture \ref{Jones conjecture}, which would be an analogue of the celebrated Serre's Open Image Theorem. Theorem \ref{thm 4} can be compared to the following standard fact in the theory of elliptic curves (see for example \cite[Theorem III.7.9, Remark III.7.10]{silverman}): if $K$ is a number field, $E/K$ is an elliptic curve, $\ell$ is a prime and $\rho_{E,\ell}\colon \gal(\overline{K}/K)\to \GL_2(\Q_\ell)$ is the Galois representation on the $\ell$-adic Tate module of $E$, then $\im(\rho_{E,\ell})$ is an abelian subgroup of $\GL_2(\Q_\ell)$ if and only if $E$ has complex multiplication.

Theorem \ref{thm 4} naturally raises the following question. 

\begin{quote}
Which topologically finitely generated closed abelian subgroups of $\Omega_{\infty}$ can be obtained as arboreal images over some global field?
\end{quote}

This has been posed for number fields in a recent preprint by Andrews and Petsche \cite{andrews}. There, the authors propose a conjecture \cite[Conjecture 1]{andrews} that characterizes pairs $(f,\alpha)$, where $f$ is a polynomial over a number field and $\alpha\in K$, such that $K_\infty(f,\alpha)$ is an abelian extension of $K$. In \cite{andrews}, they prove their conjecture for \emph{quadratic stable pairs} $(f,\alpha)$ in $\Q$; this amounts to say that $f\in \Q[x]$ is a quadratic polynomial, $\alpha\in \Q$ and $f^n-\alpha$ is irreducible for every $n$. Their proof relies on ideas from Arakelov theory and requires some computer-assisted numerical verifications. 

Here we provide a completely alternative route, which gives at once a stronger result and a one-page proof. Our approach consists of combining local considerations, building on a result of Anderson et al.\ \cite{anderson}, with Theorem \ref{thm 4}. The upshot is that we establish \cite[Conjecture 1]{andrews} for any quadratic pair $(f,\alpha)$ over $\Q$, without requiring any stability assumption. Recall that if $K$ is a field, $f,g\in K[x]$ and $\alpha,\beta\in K$, we say that the pair $(f,\alpha)$ is \emph{conjugate} to $(g,\beta)$ if there exist $u,v\in K$ with $u\neq 0$ such that $u^{-1}f(ux+v)-v/u=g(x)$ and $u\beta+v=\alpha$. Clearly, if this happens then $G_\infty(f,\alpha)\cong G_\infty(g,\beta)$. Our result (Theorem \ref{thm_abelian}) is the following.

\begin{theorem}
Let $f\in \Q[x]$ be a monic, quadratic polynomial and let $\alpha \in \Q$ be non-exceptional for $f$. Then $G_\infty(f,\alpha)$ is abelian if and only if $(f,\alpha)$ is $\Q$-conjugate to either $(x^2,\pm1)$ or $(x^2-2,\beta)$, where $\beta\in \{\pm 2,\pm 1,0\}$.
\end{theorem}

In Theorem \ref{function_fields} we switch the attention to global fields of positive odd characteristic and we completely classify quadratic pairs $(f,\alpha)$ with dynamical abelian Galois groups: we prove that they necessarily arise, up to conjugation, from \emph{constant pairs}, namely pairs $(g,\beta)$ where $\beta$ and the coefficients of $g$ belong to a finite field (and of course, every such pair has abelian dynamical Galois group).

The local considerations that we used are based on \emph{tame} ramification, where it is enough to show that the local arboreal representation has infinite ramification to rule out abelian images (given that the relevant Galois groups are pro-$2$-groups). One might wonder whether a careful examination of the wild ramification combined with the Hasse--Harf theorem, which constrains the upper ramification jumps in abelian extensions of local fields to be integers, might lead to further progress on this question and on \cite[Conjecture 1]{andrews}. We leave this possibility to future research. 

Finally, in Section \ref{sec:non_realizable} we construct infinite families of subgroups of $\Omega_\infty$ that are not topologically finitely generated and cannot be realized as arboreal images over number fields. Let us briefly explain the framework here. Recall that there is a bijection between the set of closed maximal subgroups of $\Omega_\infty$ and the set of non-zero vectors in $\bigoplus_{i\in\Z_{\geq 1}}\F_2$ (see Section \ref{sec:finite_generation} for details). Here and in the rest of the paper, when $A$ is a set of indexes and $\{G_i\}_{i \in A}$ is a collection of groups, the symbol $\bigoplus_{i \in A}G_i$ denotes the \emph{direct sum} of those groups, i.e. the group whose elements are taken to be collections $(g_i)_{i \in A}$ with $g_i \in G_i$ for each $i \in A$ and $g_i=\text{id}_{G_i}$ for all but finitely many $i$ in $A$.

For every non-zero vector $v \in \bigoplus_{i\in\Z_{\geq 1}}\F_2$, denote by $M_v$ the corresponding maximal subgroup. For any subset $J\subseteq \bigoplus_{i\in \Z_{\geq 1}}\F_2$ let $M_J\coloneqq \bigcap_{v\in J}M_v$. It is immediate to see that $[\Omega_\infty:M_J]=\infty$ if and only if $J$ is infinite. Conjecture \ref{Jones conjecture} implies then the following special case.

\begin{conjecture} \label{Jones conjecture:2}
Let $K$ be a number field, let $f\in K[x]$ be quadratic and let $J\subseteq \bigoplus_{i\in\Z_{\geq 1}}\F_2$ be an infinite subset. If $\im(\rho_{f,0})$ is contained in $M_J$, then $f$ is PCF or $0$ is periodic for $f$. \footnote{Observe that for an arboreal image $\im(\rho_{f,0})$, being contained in $M_J$ is a well-defined notion, since $M_J$ is normal in $\Omega_{\infty}$: it does not depend on the identification between the tree of roots and the ``abstract'' tree.}
\end{conjecture} 
Our next two main results establish a number of cases of Conjecture \ref{Jones conjecture:2}, hence providing unconditional support in favour of Conjecture \ref{Jones conjecture}. We pinpoint in Definition \ref{progressing definition} a certain class of subsets $J\subseteq \bigoplus_{i\in\Z_{\geq 1}}\F_2$, which we call $(k,\ell)$-progressing, where $k$ and $\ell$ are positive integers. We then prove the following theorem.

\begin{theorem} \label{thm2}
Let $k,\ell$ be positive integers and let  $J\subseteq \bigoplus_{i\in\Z_{\geq 1}}\F_2$ be a $(k,\ell)$-progressing subset.
\begin{enumerate}
\item Conjecture \ref{Jones conjecture:2} holds for $M_J$. In particular, if $\im(\rho_{f,0})\subseteq M_J$ then $f$ is PCF.
\item For every number field $K$ and every quadratic polynomial $f\in K[x]$, $\im(\rho_{f,0})\neq M_J$.
\end{enumerate}
\end{theorem}
As a special case of Theorem \ref{thm2}, we establish in Corollary \ref{evidence1: conj} that the commutator subgroup $[\Omega_{\infty},\Omega_{\infty}]$ (that corresponds to $J=\bigoplus_{i\in \Z_{\geq 1}}\F_2$) cannot be realized as an arboreal image over a number field. 

For the special subclass of polynomials of the form $f=x^2+a$ we exploit the well-known divisibility relations displayed by the elements of the post-critical orbit of $f$, which we summarize in the well-known Lemma \ref{divisibility}. Combining this lemma together with a consequence of the Capitulation Theorem (Proposition \ref{square_extension}) we are led to pinpoint in Definition \ref{M-coprimality definition} a second class of subsets $J\subseteq \bigoplus_{i\in\Z_{\geq 1}}\F_2$, which we call $M$-coprime, for some nonnegative integer $M$. For this class we are able to prove the following. 

\begin{theorem} \label{thm3}
Let $M$ be a nonnegative integer. Conjecture \ref{Jones conjecture:2} holds for every $M$-coprime subset $J\subseteq \bigoplus_{i\in\Z_{\geq 1}}\F_2$, when one restricts it to polynomials of the form $x^2+a$. 
\end{theorem}
Finally, in Corollary \ref{evidence2:conj} we use Bertrand's Postulate to construct an infinite class of examples where Theorem \ref{thm2} is not available but instead Theorem \ref{thm3} can be applied.

\section*{Acknowledgments}
We wish to thank the hospitality of the Instituto de Ciencias Matem\'aticas in Madrid in November 2019, which allowed us to make substantial progress on this work. The second author wishes to thank the Max Planck Institute of Mathematics in Bonn for its financial support, great work conditions and an inspiring atmosphere. Finally, we are grateful to Clayton Petsche for providing very useful feedback on the paper.

\section{Finitely generated arboreal representations}\label{sec:finite_generation}

Before stating and proving our first result, we need to recall some of the notation and a key result from \cite{ferra5}. Let $\Omega_\infty$ be the automorphism group of the infinite, regular, rooted binary tree. Every element $\sigma\in \Omega_\infty$ admits a so-called \emph{digital representation}, namely a decomposition $\sigma=(\ldots,\sigma_n,\ldots,\sigma_1)$ where $\sigma_n\in \F_2^{2^{n-1}}$ for every $n$. One has, for every $n\geq 1$, a continuous homomorphism $\phi_n\colon \Omega_\infty\to \F_2$ that sends $\sigma$ to the sum of the coordinates of $\sigma_n$. The natural homomorphism from $\Omega_\infty$ to its abelianization is given by $\widehat{\phi}\coloneqq \prod_{n\geq 1}\phi_n\colon \Omega_\infty\to \Omega_\infty^{\text{ab}}\cong\prod_{n\geq 1}\F_2$. Since $\Omega_\infty$ is a pro-2-group, its maximal closed subgroups are the ones having index two, and hence they are parametrized by the non-zero vectors in $\bigoplus_{n\geq 1}\F_2$. For such a vector $v=(v_n)_{n\geq 1}$ we denote by $M_v$ the corresponding maximal subgroup: this is nothing else than $\ker \left(\sum_{n\geq 1}v_n\phi_n\right)$.

Now let $K$ be a field of characteristic not 2 and let $f=(x-a)^2-b\in K[x]$. The \emph{adjusted post-critical orbit} of $f$ is the sequence $\{c_n\}_{n\geq 1}$ defined by $c_1\coloneqq -f(a)$ and $c_n\coloneqq f^n(a)$ for every $n\geq 2$. For every $\alpha\in K$ we let $c_{1,\alpha}\coloneqq c_1+\alpha$ and $c_{n,\alpha}\coloneqq c_n-\alpha$ for every $n\geq 2$. Let now $\alpha\in K$ and assume that $f^n(x)-\alpha$ is separable for every $n$, which is equivalent to asking that $c_{n,\alpha}\neq 0$ for every $n\geq 1$, i.e.\ that $\alpha$ does not belong to the post-critical orbit of $f$. For every $n\geq 1$ we denote by $K_n(f,\alpha)$ the splitting field of the polynomial $f^n(x)-\alpha$ and by $G_n(f,\alpha)$ its Galois group. Let $K_\infty(f,\alpha)\coloneqq \varinjlim_n K_n(f,\alpha)$ and $G_\infty(f,\alpha)$ the Galois group of $K_\infty(f,\alpha)$. This is the image of the arboreal representation $\rho_{f,\alpha}$ attached to $f$ with basepoint $\alpha$. 

Finally, for $a_1,\ldots,a_n\in K^*$ we denote by $\langle a_1,\ldots,a_n\rangle$ the $\F_2$-sub-vector space generated by the $a_i$'s in $K^*/{K^*}^2$.

The following proposition is stated for $\rho_{f,0}$ in \cite{ferra5}, but extending it to $\rho_{f,\alpha}$ for any $\alpha$ requires no non-trivial modifications.

\begin{proposition}{{\cite[Corollary 4.3]{ferra5}}}\label{powerful_corollary}
Let $K$ be a field of characteristic not 2, $f\in K[x]$ a monic, quadratic polynomial with adjusted post-critical orbit $\{c_n\}_{n\geq 1}$ and $\alpha\in K$ not belonging to the post-critical orbit of $f$. Let $v=(v_n)_{n\geq 1}\in\bigoplus_{n\geq 1}\F_2$. Then we have:

$$\im(\rho_{f,\alpha})\subseteq M_v\Longleftrightarrow \prod_{n\geq 1} c_{n,\alpha}^{v_n}\in K^{*2}.$$

\end{proposition}

Part of the proof of our Theorem \ref{PCF_equivalence} will require a distinction between global function fields and number fields. This is due to the fact that the Mordell conjecture, a theorem due to Faltings over
number fields and Samuel over global function fields, needs the curve to be non-isotrivial in the function field case. Recall that a curve $C$ over a global field $K$ of positive characteristic $p$ is called \emph{isotrivial} if there exists a curve $C'$ defined over $\overline{\F}_p$ such that $C$ is isomorphic to $C'$ over $\overline{K}$. For this reason we will now state and prove, in degree 2, a result that had been conjectured by Hindes in \cite[Conjecture 3.1]{hindes1} for polynomials of any degree. In order to do so, we first need two auxiliary results.

\begin{proposition}{{\cite[Proposition 3.3]{hindes1}}}\label{isotrivial_map}
Let $K$ be a global field of positive characteristic and $C_1,C_2$ be smooth curves over $K$. Suppose there exists a non-constant morphism $\varphi\colon C_1\to C_2$. If $C_1$ is isotrivial, then also $C_2$ is isotrivial.
\end{proposition}

\begin{proposition}[{{\cite[Proposition 1.2]{lockhart}}}]\label{unique_equation}
Let $K$ be a field of characteristic not 2 and $C/K$ be a hyperelliptic curve of genus $g\geq 1$. Then $C$ has a Weierstrass equation over $K$ of the form $y^2=f(x)$, where $f\in K[x]$ is a monic separable polynomial of degree $2g+1$. Moreover, such equation is unique up to a change of coordinates of the form $x=u^2x'+r$, $y=u^{2g+1}y'$ for some $u,r\in K$ with $u\neq 0$.
\end{proposition}

\begin{theorem}\label{isotriviality}
Let $K$ be a global function field of characteristic $p\neq 2$. Let $\gamma,c\in K$ be distinct elements not both in $\overline{\F}_p$ and $\phi=(x-\gamma)^2+c\in K[x]$. Suppose that $\phi^n$ is separable for every $n$. For every $m\geq 1$, let $C_m\colon y^2=\phi^m(x)$. Then the following hold.
\begin{enumerate}
\item Suppose that $\gamma$ does not equal $c+t\sqrt{-c}$ for any $t\in \overline{\F}_p$. Then the curve $C_2$ is smooth and not isotrivial.
\item Suppose that $\gamma=c+t\sqrt{-c}$ for some $t\in \overline{\F}_p^*$. Then the curve $C_3$ is smooth and not isotrivial.
\item The curve $C_m$ is smooth and not isotrivial for every $m\geq 3$.
\end{enumerate}
\end{theorem}
\begin{proof}
The smoothness part of the statements follows from the separability of $\phi^n$. Notice that for every $m,n\geq 1$ with $m\geq n$, there is a non-constant morphism $C_m\to C_n$ given by $(x,y)\mapsto (\phi^{m-n}(x),y)$. Hence (3) follows from (1), (2) and Proposition \ref{isotrivial_map}. 

$(1)$ Let $E\colon y^2=(x-c)\phi(x)$. This is a smooth curve, because the roots of $(x-c)\phi(x)$ are $c$ and $\gamma\pm\sqrt{-c}$. It follows that the only possibility for $E$ not being smooth is that $c=\gamma\pm\sqrt{-c}$, which is forbidden by hypothesis. There exists a non-constant morphism
$$C_2\to E$$
$$(x,y)\mapsto (\phi(x),y(x-\gamma))$$
Hence, if we show that $E$ is not isotrivial then the non-isotriviality of $C_2$ follows by Proposition \ref{isotrivial_map}.

Now we claim that $E$ is isotrivial if and only if there exists $t\in \overline{\F}_p$ such that $\gamma=c+t\sqrt{-c}$.

First, assume that $E$ is isotrivial. By Proposition \ref{unique_equation} this amounts to say that over $\overline{K}$, $E$ has an equation of the form $y'^2=g(x')$, where $g(x')\in \overline{\F}_p[x']$ is a monic, cubic polynomial and the two Weierstrass equations for $E$ are related by a change of coordinates of the form $x=u^2x'+r$, $y=u^3y'$ for some $u,r\in \overline{K}$ with $u\neq 0$. Again, this amounts to say that $f(u^2x'+r)/u^6=g(x')$, where $f=(x-c)\phi(x)$. But then let $x_1,x_2,x_3\in \overline{\F}_p$ be the roots of $g(x')$. These must correspond, through the aforementioned change of coordinates, to the roots of $f$, which are $c,\gamma\pm\sqrt{-c}$. Without loss of generality we can assume that $\gamma+\sqrt{-c}=u^2x_1+r$ and $\gamma-\sqrt{-c}=u^2x_2+r$. Subtracting term by term, it follows that $u^2=\frac{2\sqrt{-c}}{x_1-x_2}$. Now necessarily $c=u^2x_3+r$, and subtracting term by term the relation for $\gamma+\sqrt{-c}$ one gets that $\gamma+\sqrt{-c}-c=\frac{2\sqrt{-c}}{x_1-x_2}(x_1-x_3)$. We conclude immediately that $\gamma=c+t\sqrt{-c}$ for a constant $t\in \overline{\F}_p$.

Conversely, assume that there exists $t\in \overline{\F}_p$ such that $\gamma=c+t\sqrt{-c}$. Let $u\coloneqq \sqrt[4]{-c}$ (note that $u\neq 0$ since $0$ is not in the critical orbit of $\phi$). Then the change of coordinates $x=u^2x'+c$, $y=u^3y'$ brings the curve $E$ to the form $y'^2=x'((x'-t)^2-1)$, which is defined over $\overline{\F}_p$.

$(2)$ The proof follows the very same logic of that of part $(1)$. Let $\widetilde{C}_3\colon y^2=(x-c)\phi^2(x)$. First, we need to verify that this is a smooth curve of genus 2, which is equivalent to ask that $(x-c)\phi^2(x)$ has pairwise distinct roots. Of course $\phi^2(x)$ already has this property by hypothesis. Hence, this can only fail for $(x-c)\phi^2(x)$ if $c=\gamma\pm\sqrt{\gamma-c\pm\sqrt{-c}}$. However, this relation would imply that $(\gamma-c)^2=\gamma-c\pm\sqrt{-c}$. Now recalling that $\gamma-c=t\sqrt{-c}$ for some $t\in \overline{\F}_p^*$, we get that $t^2(\sqrt{-c})^2-(t\pm1)\sqrt{-c}=0$, which implies that $\sqrt{-c}\in \overline{\F}_p^*$. This is impossible because it would imply that $c,\gamma\in \overline{\F}_p$.

Now notice that there is a non-constant morphism $C_3\to \widetilde{C}_3$ given by $(x,y)\mapsto (\phi(x),y(x-\gamma))$. Thus by Proposition \ref{isotrivial_map} it is enough to show that $\widetilde{C}_3$ is not isotrivial. Once again, by contradiction suppose it is. Then by Proposition \ref{unique_equation} $\widetilde{C}_3$ has an equation over $\overline{K}$ of the form $y'^2=h(x')$ where $h(x')\in \overline{\F}_p[x]$ is a monic polynomial of degree 5, and the two equations are related by a change of coordinates of the form $x=u^2x'+r$, $y=u^5y'$ for some $u,r\in \overline{K}$ with $u\neq 0$. By the same logic of part $(1)$, if $x_1,\ldots,x_5\in \overline{\F}_p$ are the roots of $h(x')$ we must have, up to reordering the roots, that $\gamma+\sqrt{\gamma-c+\sqrt{-c}}=u^2x_1+r$ and $\gamma-\sqrt{\gamma-c+\sqrt{-c}}=u^2x_2+r$. These together imply that $u^2=\frac{2\sqrt{\gamma-c+\sqrt{-c}}}{x_1-x_2}$. On the other hand, $c=u^2x_3+r$ and subtracting from the relation for $x_1$ we get that
$$\gamma-c+\sqrt{\gamma-c+\sqrt{-c}}=u^2(x_1-x_3)=k\sqrt{\gamma-c+\sqrt{-c}},$$
where $k\coloneqq \frac{2(x_1-x_3)}{x_1-x_2}$. This implies that $(\gamma-c)^2=(k-1)^2(\gamma-c)+(k-1)\sqrt{-c}$. Now recall that $\gamma-c=t\sqrt{-c}$ for some $t\in \overline{\F}_p^*$, and therefore $t^2(\sqrt{-c})^2-(k-1)(t(k-1)-1)\sqrt{-c}=0$. Thus $\sqrt{-c}$ is constant and consequently $c$ and $\gamma$ are both constant, contradicting the hypotheses.
\end{proof}

It is worth noticing that Theorem \ref{isotriviality} yields a stronger conclusion than that of \cite[Conjecture 3.1]{hindes1} for quadratic polynomials. In fact, such conjecture requires $\phi$ to be non-isotrivial, which means that there are no $r\in K$ such that $\phi(x+r)-r\in \overline{\F}_p[x]$. However, we only require that $\phi$ is not of the form $(x-\gamma)^2+\gamma$ and that $\phi\notin \overline{\F}_p[x]$, which is clearly a weaker condition. Clearly this condition is also necessary: if $\phi=(x-\gamma)^2+\gamma$ or $\phi\in \overline{\F}_p[x]$ it is immediate to verify that $C_m$ is isotrivial for every $m\geq 1$.

\begin{remark}
Theorem \ref{isotriviality} immediately implies that every $\phi$ satisfying its hypotheses also satisfies those of \cite[Theorem 1.1]{hindes2}, and therefore in turn it implies, via an effective version of Mordell conjecture, a uniform bound on the Zsigmondy set for wandering points for all such polynomials. This highlights the upshot of working over global function fields: there are more conditions to control, but stronger results are available. Moreover, not only this is an interesting result on its own, but it also implies, under suitable stability assumptions, bounds on the index of the arboreal Galois representation attached to such polynomials. This can be seen by adapting the arguments used for example in \cite{ferra4,hindes3,hindes2}.
\end{remark}

We are now ready to state and prove the main result of this section.
\begin{theorem}\label{PCF_equivalence}
 Let $K$ be a global field of characteristic not 2, $f\in K[x]$ be monic and quadratic and $\alpha\in K$ not in the post-critical orbit of $f$. Then the following are equivalent.
\begin{enumerate}
\item $f$ is post-critically finite;
\item $K_\infty(f,\alpha)$ ramifies at finitely many places;
\item $\im(\rho_{f,\alpha})$ is topologically finitely generated.
\item the sub-vector space $\langle c_{1,\alpha},\ldots,c_{n,\alpha},\ldots\rangle\subseteq K^*/{K^*}^2$ is finite dimensional.
\end{enumerate}
\end{theorem}
\begin{proof}
(1)$\implies$ (2) This follows from the well-known fact that if a prime $\p$ ramifies in $K_\infty(f,\alpha)$, then $\p\mid (2)$ or there exists $n\geq 1$ such that $v_\p(c_{n,\alpha})\neq 0$. For a stronger statement, see  \cite[Lemma 2.6]{jones_divisors}.

(2)$\implies$ (3) This is proven in \cite[Theorem 3.1]{jones_survey}. Notice that for a pro-2-group, being topologically generated by finitely many conjugacy classes is equivalent to being topologically finitely generated.

(3)$\implies$ (4) Let $G\coloneqq \im(\rho_{f,\alpha})$. Since $G$ is topologically finitely generated then the group $\hom_{\text{cont}} (G,\F_2)$ is finite. The inclusion $G\hookrightarrow \Omega_\infty$ induces by restriction a map $r\colon\Omega_\infty^{\vee}\to \hom_{\text{cont}} (G,\F_2)$. Since $\Omega_\infty^{\vee}$ is generated by all the $\phi_n$'s, it follows that there exists some $n_0\in \N$ such that for every $n>n_0$, $r(\phi_n)=r(\phi_m)$ for some $m\leq n_0$. This amounts to say that for every $n>n_0$ there exists $m\leq n_0$ such that $G\leq M_{v_{m,n}}$, where $v_{m,n}\in\bigoplus_{\Z_{\geq 1}}\F_2$ is the vector having 1 precisely at the $m$-th and $n$-th entry, and 0 elsewhere. But then, by Proposition \ref{powerful_corollary}, it follows that $c_{m,\alpha}c_{n,\alpha}\in K^{*2}$. This clearly shows that $\{c_{1,\alpha},\ldots,c_{n_0,\alpha}\}$ generates $\langle c_{1,\alpha},\ldots,c_{n,\alpha},\ldots\rangle$.

(4)$\implies$(1) Suppose by contradiction that $f$ is not post-critically finite. Notice that $\{c_{n,\alpha}\}$ is the adjusted post-critical orbit of the polynomial $g\coloneqq f(x+\alpha)-\alpha\in K[x]$. Hence, $g$ is also not post-critically finite. Now choose $n\geq 3$ such that $\{c_{1,\alpha},\ldots,c_{n,\alpha}\}$ generates $\langle c_{i,\alpha}\rangle_{i\in \N}$, and let $C_n\colon y^2=g^n(x)$. This is a smooth curve, because $c_{i,\alpha}\neq 0$ for every $i$, and has genus at least 2. If $K$ is a global function field, notice moreover that since $g$ is not post-critically finite, then it fulfils the hypotheses of Theorem \ref{isotriviality}. It follows that $C_n$ is a non-isotrivial curve. Since $\{c_{1,\alpha},\ldots,c_{n,\alpha}\}$ generates $\langle c_{i,\alpha}\rangle_{i\in \N}$, by pigeonhole there exist $e_1,\ldots,e_n\in \{0,1\}$  such that $c_{1,\alpha}^{e_1}\ldots c_{n,\alpha}^{e_n}c_{m,\alpha}$ is a square in $K$ for infinitely many $m$'s. Then there are infinitely many $m$'s such that in the global field $L\coloneqq K\left(\sqrt{c_{1,\alpha}^{e_1}\ldots c_{n,\alpha}^{e_n}}\right)$ the element $c_{m,\alpha}$ is a square. Since $g$ is post-critically infinite, this implies that $C_n$ has infinitely many $L$-rational points, contradicting Mordell's conjecture.
\end{proof}
\begin{remark}
Observe that the implication $(3) \implies (1)$ in Theorem \ref{PCF_equivalence} is clearly false for a general field $K$. Indeed there are several fields with finitely generated absolute Galois group and admitting non post-critically finite quadratic polynomials, e.g.\ algebraically closed fields of characteristic $0$ or local fields.

On the other hand, from the proof above it is clear that the implication $(3)\implies (4)$ holds for any ground field of characteristic not 2.
\end{remark}

\section{Abelian dynamical Galois groups}\label{sec:abelian}

In this section we focus on quadratic polynomials having abelian dynamical Galois group. The question has been raised in a recent preprint by Andrews and Petsche \cite{andrews}. There, the authors propose a conjecture \cite[Conjecture 1]{andrews} that characterizes pairs $(f,\alpha)$, where $f$ is a polynomial over a number field $K$ and $\alpha\in K$, such that $K_\infty(f,\alpha)$ is an abelian extension of $K$. In \cite{andrews}, they prove their conjecture for quadratic \emph{stable} pairs $(f,\alpha)$ in $\Q$; this amounts to say that $f$ is a quadratic polynomial over $\Q$, $\alpha\in \Q$ and $f^n-\alpha$ is irreducible for every $n$.

Here, we will first show that if $K$ is a global field of characteristic not 2, $\alpha\in K$ and $f\in K[x]$ is a quadratic polynomial such that $G_\infty(f,\alpha)$ is abelian, then $f$ is PCF. Next, we show that \cite[Conjecture 1]{andrews} holds for any quadratic pair $(f,\alpha)$ over $\Q$. Notice that we do not require any stability assumption. Moreover, our proof technique differs heavily from that of \cite{andrews}: ours is based on exploiting, through local class field theory, a result of Anderson et al.\ \cite{anderson}.

Finally, as another application of our method, we will prove that when $K$ is a global function field of odd characteristic, the only quadratic pairs $(f,\alpha)$ yielding abelian dynamical Galois groups are the ones coming up to conjugation from constant pairs, i.e.\ pairs $(g,\beta)$ such that $\beta$ and the coefficients of $g$ are all constant. 

As a first thing, we need the following proposition from \cite{ferra5}. Recall that the continuous homomorphism $\phi_1\colon \Omega_\infty\to \F_2$ is defined as $\phi_1(\sigma)=0$ if and only if $\sigma$ does not swap the two halves of the tree. 

\begin{proposition}{{\cite[Proposition 7.3]{ferra5}}}\label{commutativity}
Let $\sigma,\tau\in \Omega_\infty$ be such that $\phi_1(\tau)=1$ and $\sigma$ is not in $\{1,\tau\}$ modulo commutators. Then $\sigma$ and $\tau$ do not commute.
\end{proposition}
We now apply this proposition to show that closed abelian subgroups of $\Omega_{\infty}$ have infinite index therein. We first make a definition which will be used to obtain a rather more precise conclusion. For every $n\geq 1$, we denote by $\Omega_n$ the automorphism group of the infinite, regular, rooted binary tree truncated at level $n$.
\begin{definition} \label{Level of a subgroup}
Let $G$ be a closed non-trivial subgroup of $\Omega_{\infty}$. We call the \emph{level} of $G$ the smallest non-negative integer $n$ such that the image of $G$ in $\Omega_{n+1}$ via the canonical projection $\Omega_\infty\to \Omega_{n+1}$ is non-trivial. 
\end{definition}
The following simple remark will play an important role.
\begin{remark} \label{first non-trivial}
Let $G$ be a closed non-trivial subgroup of $\Omega_{\infty}$ and let $n$ be its level. Let $V_n$ be the set of nodes at distance $n$ from the root of the tree. Observe that, by construction, $G$ preserves each of the $2^n$ trees rooted in the nodes of $V_n$. For each $w$ in $V_n$, we denote by $\rho_w:G \to \Omega_{\infty}$ the homomorphism induced by the identification between $T_{\infty}$ and $wT_{\infty}$ given by $v \mapsto wv$. By definition of level, there must be at least one $w_0$ in $L_n$ such that $\phi_1 \circ \rho_{w_0}(G) \neq \{0\}$. We call such a $w_0$ a \emph{faithful} node for $G$. 
\end{remark}
\begin{proposition} \label{abelian subgroups are of infinite index}
Let $G$ be an abelian closed subgroup of $\Omega_{\infty}$. Then we have the following. 
\begin{enumerate}
\item If $G$ is non-trivial and of level $n$, then for each faithful node $w_0$ for $G$ in $V_n$ we have that $\rho_{w_0}(G)$ has $1$-dimensional image in $\Omega_{\infty}^{\emph{ab}}$. 
\item $G$ has infinite index in $\Omega_{\infty}$. 
\end{enumerate} 

\end{proposition}
\begin{proof}
Part $(1)$ is an immediate consequence of Proposition \ref{commutativity}. Part $(2)$ is obvious if $G$ is the trivial group and follows immediately from part $(1)$ otherwise. 
\end{proof}
Let us now recall the definition of exceptional point.
\begin{definition} \label{exceptional point def}
Let $K$ be a field and let $\overline{K}$ be an algebraic closure. Let $f\in K[x]$ and $\alpha \in \overline{K}$. We say that $\alpha$ is \emph{exceptional} for $f(x)$ if the backward orbit $\bigcup_{n\geq 1}f^{-n}(\alpha)$ is a finite subset of $\overline{K}$.  
\end{definition}

Before stating the next lemma, let us remark the following elementary fact that will be used in the proof.
\begin{remark}\label{subquotient}

Let $f\in K[x]$, $\alpha\in K$ and $\beta \in f^{-n}(\alpha)$ for some $n$. The root $\beta$ belongs to a finite extension $L$ of $K$. Then the dynamical Galois group $G_\infty(f,\beta)$ (which is the inverse limit of the Galois groups of $f^n-\beta$ over $L$) is a subquotient of $G_\infty(f,\alpha)$: the Galois group $\gal(K_\infty(f,\alpha)/L)$ is a subgroup of $G_\infty(f,\alpha)$ admitting $G_\infty(f,\beta)$ as a quotient, by Galois theory.
\end{remark}

\begin{lemma}\label{finiteness}
Let $K$ be a global field of characteristic not 2, $f\in K[x]$ a monic, quadratic polynomial and $\alpha\in \overline{K}$.
\begin{enumerate}[(a)]
\item If $G_\infty(f,\alpha)$ is finite, then $\alpha$ belongs to the post-critical orbit of $f$.
\item If $f$ is post-critically infinite, then $G_\infty(f,\alpha)$ is infinite.
\item $G_\infty(f,\alpha)$ is finite if and only if $\alpha$ is exceptional.
\end{enumerate}
\end{lemma}
\begin{proof}
$(a)$ Up to replacing $K$ with a finite extension, we can assume that $G_\infty(f,\alpha)$ is trivial. Then $f^n-\alpha$ splits into linear factors for every $n$. Thus for all but finitely many places $\mathfrak{p}$ of $K$,\footnote{We exclude the places at which some coefficient of $f$ or $\alpha$ have negative valuation plus the Archimedean places in the number field case} there exists some sufficiently large $n\in \mathbb{Z}_{\geq 1}$ such that the polynomial $f^n-\alpha$ has two coincident roots modulo $\mathfrak{p}$. Hence, by \cite[Lemma 2.6]{jones_divisors}, $\alpha$ is in the post-critical orbit of $f$ modulo $\mathfrak{p}$ for all but finitely many places $\mathfrak{p}$ of $K$. By \cite[Theorem 3.1]{benedetto1} when $K$ is a number field and by \cite[Theorem 3]{towsley} when $K$ is a function field, this implies that $\alpha$ is in the post-critical orbit of $f$.

$(b)$ Let $f=(x-a)^2-b$ be post-critically infinite, and suppose by contradiction that $G_\infty(f,\alpha)$ is finite. Then all of its subquotients are finite, and hence by $(a)$ and Remark \ref{subquotient} this means that the whole backward orbit $\bigcup_{n\geq 1}f^{-n}(\alpha)$ is contained in the post-critical orbit of $f$. Now choose an $n$ such that $f^{-n}(\alpha)$ contains at least two distinct elements $\beta_1,\beta_2$. Notice that this is always possible for $n=1$ or $2$: if it is not possible for $n=1$, it means that $\alpha=-b$. If $\alpha=-b$ and $f^{-2}(-b)$ has just one element, then $f=(x-a)^2+a$, which is post critically finite (notice that in fact $a$ is an exceptional point for such polynomial). Then there are $m_1,m_2$ such that $\beta_1=f^{m_1}(a)$ and $\beta_2=f^{m_2}(a)$. But this implies that $\alpha=f^{m_1+n}(a)=f^{m_2+n}(a)$, implying $m_1=m_2$ because $f$ is post-critically infinite, and yielding a contradiction because $\beta_1\neq \beta_2$ by construction.

$(c)$ One direction is obvious by definition. Conversely, assume that $G_\infty(f,\alpha)$ is finite. Again, it follows from point $(a)$ and Remark \ref{subquotient} that the set of roots $\bigcup_{n\geq 1}f^{-n}(\alpha)$ is entirely contained in the post-critical orbit of $f$. Furthermore thanks to point $(b)$ we know that $f$ must be PCF. Hence $\bigcup_{n\geq 1}f^{-n}(\alpha)$ is finite, which means exactly that $\alpha$ is exceptional.
\end{proof}

Notice that points $(b)$ and $(c)$ of Lemma \ref{finiteness} together imply that only PCF polynomials can admit exceptional points.

We are now ready to show that only PCF polynomials can have abelian dynamical Galois groups.

\begin{theorem}\label{abelian_image}
Let $K$ be a global field of characteristic not 2, let $f\in K[x]$ be a monic, quadratic polynomial and let $\alpha\in K$. Let $G_\infty(f,\alpha)$ be the dynamical Galois group of the system. If $G_\infty(f,\alpha)$ is abelian, then $f$ is PCF.
\end{theorem}
\begin{proof}

By contradiction, assume that $f$ is post-critically infinite. Then by Lemma \ref{finiteness} the group $G_\infty(f,\alpha)$ is infinite. It follows that, in particular, $f^m-\alpha$ has a root $\beta\notin K$ for some $m$. Since the post-critical orbit of $f$ is entirely contained in $K$, we conclude that $f^n-\beta$ is separable for every $n\geq 1$, and therefore the group $G_\infty(f,\beta)$ is the image of the arboreal representation $\rho_{f,\beta}$ over $K(\beta)$. By Remark \ref{subquotient}, $G_\infty(f,\beta)$ is a subquotient of $G_\infty(f,\alpha)$, and hence it is abelian again. Moreover, by Lemma \ref{finiteness} again, $G_\infty(f,\beta)$ is not trivial. Let $\gamma$ be a faithful node for $G_\infty(f,\beta)$ acting on the tree of roots of the collection of polynomials $f^n-\beta$'s: by Proposition \ref{abelian subgroups are of infinite index} the group $G_\infty(f,\gamma)$ has 1-dimensional image in $\Omega_\infty^{\text{ab}}$ (notice that $\rho_{f,\gamma}(G_\infty(f,\beta))=G_\infty(f,\gamma))$.  It follows by Proposition \ref{powerful_corollary} that $\langle c_n-\gamma\rangle_{n\geq 1}$ is finite dimensional modulo squares, contradicting Theorem \ref{PCF_equivalence}.
\end{proof}

Our next goal is to use Theorem \ref{abelian_image} in order to prove \cite[Conjecture 1]{andrews} in the quadratic case over $\Q$. To do so, let us first recall the following local results.

\begin{definition}
Let $K$ be a local field and $L/K$ be a separable extension. We say that $L/K$ is \emph{infinitely ramified} if for every integer $N$ there exists a finite sub-extension $F/K$ in $L/K$ whose ramification index satisfies the lower bound $e_{F/K} \geq N$.
\end{definition}
 
\begin{theorem}[{{\cite[Lemma 7.1, Theorem 7.4]{anderson}}}]\label{poonen}
Let $k$ be a local field of odd residue characteristic with ring of integers $\O_k$ and maximal ideal $\m$. Let $f=x^2+c\in k[x]$ and $\alpha\in k$. Let $v(c)\geq 0$ and suppose that one of the following two conditions is satisfied:
\begin{enumerate}[(a)]
\item $v(\alpha)<0$;
\item $v(\alpha)\geq 0$, $0$ is periodic modulo $\m$, $\alpha$ is in the orbit of $0$ modulo $\m$ but $\alpha$ is not in the orbit of $0$;
\end{enumerate}
Then $k_\infty(f,\alpha)/k$ is infinitely ramified.
\end{theorem}
The following result will play a key role for us. We give a quick proof using local class field theory. However we remark that it is also possible to prove it by more elementary methods, using the explicit description of tame extensions of local fields via radicals and the finiteness of the group of roots of unity in any non-archimedean local field. 
\begin{lemma}\label{class_field}
Let $k$ be a local field of odd residue characteristic $p$ and let and $k'/k$ be an infinitely ramified Galois $2$-extension. Then $k'/k$ is non-abelian.
\end{lemma}
\begin{proof}
Let $\O_k$ be the ring of integers of $k$ and $\m$ its maximal ideal.  By local class field theory (see for example \cite[Chapter V]{neukirch2}) the group $G_k^{\text{ab}}$ is isomorphic, as a topological group, to $\widehat{\Z}\times \O_k^{*}$, with the inertia group $I_{k'/k}$ admitting a continuous surjection from $\O_k^{*}\cong  (\O_k/\m)^{*}\times (1+\m)$. Since $1+\m$ is a $\Z_p$-module and $I_{k'/k}$ is a pro-2-group, such a surjection must factor through $(\O_k/\m)^{*}$, and hence the image is finite.
\end{proof}
We are now in place to prove the main results of this section. 
\begin{theorem}\label{thm_abelian}
Let $f\in \Q[x]$ be a monic, quadratic polynomial and let $\alpha \in \Q$ be non-exceptional for $f$. Then $G_\infty(f,\alpha)$ is abelian if and only if $(f,\alpha)$ is $\Q$-conjugate to either $(x^2,\pm1)$ or $(x^2-2,\beta)$, where $\beta\in \{\pm 2,\pm 1,0\}$.
\end{theorem}
\begin{proof}
First, suppose that $f$ is conjugated to $x^2$ or $x^2-2$. These are somewhat special cases, as they are the powering map and the Chebyshev map, respectively, of degree 2. Then, as shown in \cite[Theorems 12,13]{andrews}, it follows elementarily from a theorem of Amoroso and Zannier \cite{amoroso} that $G_\infty(f,\alpha)$ is abelian if and only if $f=x^2$ and $\alpha$ is a root of unity or $f=x^2-2$ and $\alpha=\zeta+1/\zeta$, where $\zeta$ is a root of unity. Since $\alpha\in \Q$, in the first case one must have $\alpha=\pm 1$, while in the second case $\zeta$ must satisfy a quadratic equation over the rationals, and hence $\zeta\in \{\pm 1,\pm i,\pm 1/2\pm i\sqrt{3}/2\}$, yielding the conclusion.

Now assume that $G_\infty(f,\alpha)$ is abelian. Observe that up to conjugation over $\Q$, we can assume that the pair $(f,\alpha)$ is of the form $(x^2+c,\alpha)$ for some $c,\alpha\in \Q$. By Theorem \ref{abelian_image}, it follows that $f$ is post-critically finite. It is easy to show (see for example \cite{boston}) that there are exactly three post-critically finite polynomials in $\Q[x]$ of the aforementioned form: they are $x^2$, $x^2-1$ and $x^2-2$. If $f\in \{x^2,x^2-2\}$, the claim follows as we explained above.

Hence we are left with the case $f=x^2-1$, and we need to show that $G_\infty(f,\alpha)$ is non-abelian for every $\alpha\in \Q$. First, let us assume that $\alpha\notin \{0,-1\}$. If there exists an odd prime $p$ such that $v_p(\alpha)\neq 0$, then we conclude immediately by Theorem \ref{poonen} and Lemma \ref{class_field}, since if $\im(\rho_{f,\alpha})$ is abelian then so is $\im({\rho_{f,\alpha}}_{|G_{\Q_p}})$.\footnote{Here we are implicitly fixing an embedding $\overline{\Q}\hookrightarrow \overline{\Q}_p$.} This immediately implies that $\alpha$ must be of the form $\pm 2^j$ for some $j\in \Z$. Suppose first that $\alpha\notin\{ 1,-2,-1/2\}$. Notice that $\Q(\sqrt{\alpha+1})_\infty(f,\sqrt{\alpha+1})\subseteq \Q_\infty(f,\alpha)$, and therefore it is enough to show that $\Q(\sqrt{\alpha+1})_\infty(f,\sqrt{\alpha+1})$ is a non-abelian extension of $\Q(\sqrt{\alpha+1})$. But this follows by the same argument we just used: for these values of $\alpha$ there always exists an odd prime $p$ such that $v_p(\alpha+1)\neq 0$, and therefore Theorem \ref{poonen} and Lemma \ref{class_field} allow to conclude.

Next, we have to exclude the cases $\alpha\in \{1,-2,-1/2\}$. As usual, we denote by $\{c_n\}_{n\geq 1}$ the adjusted post-critical orbit of $f$ and by $c_{n,\alpha}$ the sequence defined by $c_{1,\alpha}\coloneqq c_1+\alpha$ and $c_{n,\alpha}\coloneqq c_n-\alpha$ for $n\geq 2$. It is a well-known fact (see for example \cite{stoll}) that if $\dim\langle c_{1,\alpha},c_{2,\alpha}\rangle=2$, then $\gal(f^2-\alpha)\cong D_8$. This immediately rules out $\alpha=1,-2$. If $\alpha=-1/2$ then $(c_{1,\alpha},c_{2,\alpha},c_{3,\alpha})=(1/2,1/2,-1/2)$. But then $\phi_1(G_\infty(f,\alpha))\neq \{0\}$, since $1/2$ is not a  rational square, and hence $-1/2$ is a faithful node for $G_\infty(f,-1/2)$. Furthermore by Proposition \ref{powerful_corollary} the image of $G_\infty(f,-1/2)$ via $\widehat{\phi}\colon \Omega_\infty\to \prod_{n\geq 1}\F_2$ is at least 2-dimensional, since $1/2$ and $-1/2$ are independent modulo rational squares. Hence by Proposition \ref{abelian subgroups are of infinite index} we conclude that the group $G_\infty(f,-1/2)$ cannot be abelian.

To conclude the proof, we have to rule out the case $\alpha\in \{0,-1\}$. Since the equation $f=-1$ has only $0$ as a solution, it follows immediately that $G_\infty(f,-1)=G_\infty(f,0)$. Thus it is enough to deal with $\alpha=0$. The roots of $f^2$ are $0$ and $\pm \sqrt{2}$ and therefore by Remark \ref{subquotient} it is enough to show that $\im (\rho_{f,\sqrt{2}})$ is not abelian. One can argue again as above: $c_{1,\sqrt{2}}=1+\sqrt{2}$ and $c_{2,\sqrt{2}}=-\sqrt{2}$. It is immediate to verify that $1+\sqrt{2}$ and $-\sqrt{2}$ are linearly independent modulo squares in $\Q(\sqrt{2})^{*}$, and therefore the Galois group of $f^2-\sqrt{2}$ over $\Q(\sqrt{2})$ is isomorphic to $D_8$.
\end{proof}
Let us remark that one can remove the non-exceptionality hypothesis on $\alpha$, at the cost of adding $(x^2,0)$ to the list. However, pairs $(f,\alpha)$ with $\alpha$ exceptional are not so interesting from the Galois-dynamical point of view: by Lemma \ref{finiteness}, up to passing to a finite extension of the base field their dynamical Galois group is trivial.
\begin{theorem}\label{function_fields}
Let $K$ be a global function field of odd characteristic $p$, $f\in K[x]$ a monic, quadratic polynomial and $\alpha\in K$. Then $G_\infty(f,\alpha)$ is abelian if and only if $(f,\alpha)$ is $K$-conjugate to a pair $(g,\beta)$ with $g\in \overline{\F}_p[x]$ and $\beta\in \overline{\F}_p$.
\end{theorem}
\begin{proof}
The ``if'' part is obvious, since the group $\gal(\overline{\F}_p/\F_p)$ is abelian.

Conversely, assume that $G_\infty(f,\alpha)$ is abelian. Up to conjugation, we can assume that $f=x^2+c$ for some $c\in K$. By Theorem \ref{abelian_image}, $f$ must be post-critically finite. We claim that this forces $c$ to be in $\overline{\F}_p$. To see this either observe that if there is a valuation $v$ of $K$ with $v(c)<0$, then we clearly have that the valuations of the orbit of $0$ go to $-\infty$, hence in particular the orbit is infinite. Alternatively the conclusion follows from the fact that if $x^2+c$ is PCF then $c$ satisfies a non-trivial polynomial equation with coefficients in $\F_p$.  The same conclusion on $f$ could have been reached using \cite[Corollary 6.3]{benedetto2}.

 Now if $\alpha$ is non-constant, then there is necessarily a place $\p$ of $K$ such that $v_\p(\alpha)<0$. Let then $K_\p$ be the completion of $K$ at $\p$: by Theorem \ref{poonen} the extension $(K_\p)_\infty(f,\alpha)/K_\p$ is infinitely ramified, and Lemma \ref{class_field} shows that it is non-abelian. It follows that $\alpha$ must be constant.
\end{proof}

\section{Non-realizable subgroups of \texorpdfstring{$\Omega_\infty$}{}}\label{sec:non_realizable}

The goal of this section is to show that if $I\subseteq \bigoplus_{i\geq 1}\F_2$ is an infinite set of vectors satisfying certain conditions, then the group $M_I\coloneqq\bigcap_{v\in I} M_v$ (cf.\ Section \ref{sec:finite_generation} for the notation) is never the image of the arboreal representation attached to a quadratic polynomial over a number field. As we will show in Lemma \ref{finite_generation}, if $I\subseteq \bigoplus_{i\geq 1}\F_2$ is \emph{any} infinite subset, then $M_I$ has infinite index in $\Omega_\infty$ and it is \emph{not} topologically finitely generated. Hence, it follows from Theorem \ref{PCF_equivalence} and Jones' Conjecture \ref{Jones conjecture} that if $K$ is a number field and $f\in K[x]$ is quadratic, then $\im(\rho_{f,0})=M_I$ implies that $0$ is periodic for $f$. Notice that under these assumptions, within all subgroups of the form $M_I$ there are certainly infinitely many that cannot be realized as arboreal images by any quadratic polynomial: if $0$ is periodic then $f$ must be reducible, and therefore by Proposition \ref{powerful_corollary} it must be $\im(\rho_f)\subseteq M_{(1,0,\ldots,0,\ldots)}$; consequently if $(1,0,\ldots,0,\ldots)$ does not belong to the sub-vector space spanned by $I$ then $M_I$ is not realizable as an arboreal image. Although $\Omega_\infty$ contains of course many more subgroups than the ones of the form $M_I$, ruling out the ones in the latter form from the possible images of arboreal representations is already a non-trivial problem, as it yields unconditional evidence towards Conjecture \ref{Jones conjecture}. In this section we rule out subgroups of the form $M_I$ for two distinct infinite families of $I$'s that may as well contain $(1,0,\ldots,0,\ldots)$. The strategy of the proof, in both cases, is to associate to a polynomial whose arboreal representation has image $M_I$ a hyperelliptic curve having infinitely many rational points, violating this way Faltings' theorem when the base field is a number field. While for the first family this is done geometrically by looking at a ``chain'' of hyperellitic curves and non-constant maps, for the second family the proof is much more delicate: the hyperelliptic curves involved are, in general, unrelated one to each other and the construction is of an entirely arithmetic nature.

We begin with a simple fact, which is a profinite variation of Schreier's Lemma. 
\begin{proposition} \label{Profinite Schreier's lemma}
Let $G$ be a profinite group and $H$ be an open subgroup. Suppose that $G$ is topologically finitely generated. Then also $H$ is topologically finitely generated. 
\begin{proof}
Let $S$ be a finite set of topological generators of $G$. Denote by $G/H$ the set of right cosets for $H$ in $G$. Fix $f:G/H \to G$ a set-theoretic section of the projection $G \to G/H$. Denote by $R$ the image of $f$. Since $H$ is open in $G$, then $R$ is a finite set. Hence the subset of $H$ given by
$$\widetilde{S}:=\{rsf(rs)^{-1}, r \in R, s \in S\}
$$
is finite. To conclude the proof, it is enough to show that $\widetilde{S}$ is a set of topological generators of $H$. Notice that this is equivalent to proving that for any open normal subgroup $N$ of $H$ the image of $\widetilde{S}$ modulo $N$ generates $H/N$, because since $H$ is profinite as well, such subgroups $N$ constitute a system of fundamental neighbors of the identity \cite[Proposition 1.1.3]{neukirch}. Observe that any open normal subgroup $N$ of $H$ is an open subgroup of $G$ as well and hence has finite index therein. Therefore it admits finitely many $G$-conjugates, whose intersection is a subgroup $N' \leq N$ that is an open normal subgroup of $G$ (and of $H$ as well, of course). If $\widetilde{S}$ generates modulo $N'$ then it certainly does modulo $N$, so we are reduced to the case of $N'$. This last case follows from the ordinary Schreier's Lemma (see \cite{hall}) applied to the group $G/N'$, subgroup $H/N'$ and generating set $S$: in this case the lemma gives precisely that $[\widetilde{S}]_{\text{mod } N'}$ is a generating set for $H/N'$.
\end{proof}
\end{proposition}
\begin{lemma}\label{finite_generation}
The following hold.
\begin{enumerate}
\item The commutator subgroup $[\Omega_\infty,\Omega_\infty]\leq \Omega_\infty$ is not topologically finitely generated.
\item Let $I\subseteq \bigoplus_{i\in \Z_{\geq 1}}\F_2$ be non-empty, and let $M_I\coloneqq \bigcap_{v\in I}M_v$. Then $M_I$ is not topologically finitely generated.
\end{enumerate}
\end{lemma}
\begin{proof}
(1) Suppose by contradiction that $[\Omega_\infty,\Omega_\infty]$ is topologically finitely generated. Then the group $\hom_{\text{cont}}([\Omega_\infty,\Omega_\infty],\F_2)$ is finite. However, as proved in \cite{ferra5}, for every $n\geq 1$ there exists a character $\widetilde{\phi}_n(x)\colon [\Omega_\infty,\Omega_\infty]\to \F_2$ whose value on $\sigma=(\sigma_n)_{n\geq 1}\in [\Omega_\infty,\Omega_\infty]\subseteq \Omega_\infty$ is given by the sum of the first half of the coordinates of $\sigma_n$. The family $\{\widetilde{\phi}_n(x)\}_{n\geq 1}$ is linearly independent and infinite, and we have a contradiction.

(2) Assume by contradiction that $M_I$ is topologically finitely generated. Then $M_I^{\text{ab}}$ is a finitely generated $\Z_2$-module. Dualizing, it follows that $(M_I^{\text{ab}})^{\vee}$ embeds into $(\Q_2/\Z_2)^m$ for some $m\geq 1$. Hence, its 2-torsion is at most $m$-dimensional as a vector space over $\F_2$. Now consider the restriction map $r\colon (\Omega_\infty^{\text{ab}})^{\vee}\to (M_I^{\text{ab}})^{\vee}$. The image of $r$ lies inside $(M_I^{\text{ab}})^{\vee}[2]$, and hence the subspace of $(M_I^{\text{ab}})^{\vee}$ spanned by $\langle \phi_{i|M_I}\rangle_{i\geq 1}$ is finite dimensional. It follows that $H\coloneqq\bigcap_{i\geq 1}\ker \phi_{i|M_I}$ is open in $M_I$. By Proposition \ref{Profinite Schreier's lemma}, this shows that $H$ is topologically finitely generated. On the other hand, it is immediate to see that $H=[\Omega_\infty,\Omega_\infty]$, and we get a contradiction from (1).
\end{proof}
We now give a definition that has the purpose of pinpointing a first infinite class of subsets $J\subseteq \bigoplus_{i\in\Z_{\geq 1}}\F_2$ , for which we will establish that $M_J$ does not arise as an arboreal Galois image over number fields. 
\begin{definition} \label{progressing definition}
Let $J\subseteq \bigoplus_{i\in\Z_{\geq 1}}\F_2$. Let $k,\ell$ be in $\Z_{\geq 1}$. We say that $J$ is $(k,\ell)$-\emph{progressing} if the sub-vector space spanned by $J$ contains an infinite set $I$ with the following property: for every $v \in I$, the indexes of the non-zero entries of $v$ form an arithmetic progression of common difference $k$ and length $\ell$.
\end{definition}
\begin{theorem} \label{evidence1}
Let $J\subseteq \bigoplus_{i\in \Z_{\geq 1}}\F_2$ be a $(k,\ell)$-progressing subset for some positive integers $k,\ell$. Let $K$ be a number field, $f\in K[x]$ a quadratic polynomial and $\alpha\in K$ an element  of $K$ not belonging to the post-critical orbit of $f$. Then we have the following. 

\begin{enumerate}[(a)]
\item If $\im(\rho_{f,\alpha}) \subseteq M_J$ then $f$ is post-critically finite and $\im(\rho_{f,\alpha})$ is finitely generated. 
\item $\im(\rho_{f,\alpha})$ cannot coincide with $M_J$. 
\end{enumerate}

\end{theorem}	
\begin{proof}
$(a)$ We only need to prove that the polynomial $f$ must be post-critically finite, by Theorem \ref{PCF_equivalence}. Let $I\subseteq J$ be the infinite set whose existence is guaranteed by the definition of $(k,\ell)$-progressing set. Up to replacing $I$ by an infinite subset, we can assume that there exists some $i_0\in \Z_{\geq 3}$ such that $v_i=0$ for every $v\in I$ and every $i<i_0$. For every $v\in I$, let $r_v$ be the index of the first non-zero entry of $v$. Since $\im(\rho_{f,\alpha})\subseteq M_v$ for every $v\in I$, Proposition \ref{powerful_corollary} implies that
\begin{equation}\label{critical_relations}
\prod_{j=r_v}^{k\ell+r_v}c_{j,\alpha}\in K^{*2} \mbox{ for every } v\in I.
\end{equation}
Now consider the curve $C\colon y^2=\prod_{j=1}^\ell f^{kj+i_0}(x)-\alpha$. This is smooth because $\alpha$ is not in the post-critical orbit of $f$ and has genus $\geq 2$ because $i_0\geq 3$. On the other hand, if $f$ were post-critically infinite with critical point $\gamma$, relation \eqref{critical_relations} would imply that $\left\{\left(f^{r_v-i_0}(\gamma),\sqrt{\prod_{j=r_v}^{k\ell+r_v}c_{j,\alpha}}\right)\right\}_{v\in I}$ is an infinite set of $K$-rational points of $C$, contradicting Faltings' theorem.

$(b)$ Suppose by contradiction that  $\im(\rho_{f,\alpha})=M_J$. The group $\im(\rho_{f,\alpha})$ is topologically finitely generated by $(a)$, but on the other hand the group $M_J$ is topologically infinitely generated by Lemma \ref{finite_generation}. 
\end{proof}
As a special case we obtain the following corollary. 
\begin{corollary} \label{evidence1: conj}
The subgroup $[\Omega_{\infty},\Omega_{\infty}]$ is not realizable as an arboreal image over any number field.
\end{corollary}
\begin{proof}
This subgroup is obtained taking $J=\bigoplus_{i\in \Z_{\geq 1}}\F_2$, which is $(k,\ell)$-progressing for any positive integers $k,\ell$. Hence the desired conclusion follows at once from Theorem \ref{evidence1} part $(b)$. 
\end{proof}
There are several subsets $J\subseteq \bigoplus_{i\in\Z_{\geq 1}}\F_2$ that are not $(k,\ell)$-progressing for any positive integers $k,\ell$. Our next definition has the purpose of pinpointing a second class of subsets $J$ for which we are able, for the sub-class of polynomial of the form $x^2+a$, to prove the analogue of Theorem \ref{evidence1} and Corollary \ref{evidence1: conj}. To motivate this definition we remind the reader of the rigid divisibility property among the elements of the post-critical orbit for such polynomials: this is given in Lemma \ref{divisibility}. 
\begin{definition} \label{M-coprimality definition}
Let $M$ be a nonnegative integer. Let $J\subseteq \bigoplus_{i\in\Z_{\geq 1}}\F_2$. We say that $J$ is \emph{$M$-coprime} if the sub-vector space spanned by $J$ contains an infinite subset $I$ such that for every $v\in I$ there exists a non-zero entry $v_i$ with $i>M$ such that the following two conditions hold:
\begin{enumerate}
\item if $v_j=1$, $j\neq i$ and $j>M$, then $(i,j)=1$;
\item as $v$ ranges over $I$, the index $i$ is not bounded from above.
\end{enumerate}
\end{definition}
For a number field $K$, we denote by $\Sigma_K$ the set of finite places of $K$. To take advantage of the divisibility properties we need the following basic fact that allows, after a finite extension, to turn elements with everywhere even valuation into squares.  
\begin{proposition}\label{square_extension}
Let $K$ be a number field. Then there exists a finite extension $L/K$ with the following property. For each $\alpha \in K^{*}$ such that $v(\alpha)$ is even for all $v \in \Sigma_K$, we have that $\alpha \in L^{*2}$. 
\begin{proof}
Let $H_K$ be the Hilbert class field of $K$. This is the largest abelian unramified extension of $K$, and it is a finite extension. Recall that by the Capitulation Theorem (see for example \cite[VI.7.5]{neukirch2}) every ideal of $K$ becomes principal in $H_K$. Furthermore, thanks to Dirichlet's unit theorem, there exists a finite extension $L/H_K$ such that all elements of $\mathcal{O}_{H_K}^{*}$ are squares in $L$. We claim that this $L$ satisfies the property described in the statement. Observe that $\alpha$ in $K^{*}$ has $v(\alpha) \equiv 0 \ \text{mod} \ 2$ for each $v \in \Sigma_K$ if and only if the fractional principal ideal $(\alpha)$ is a square in the group of fractional ideals of $K$. It follows that the extension of $(\alpha)$ to $H_K$  is the square of a principal fractional ideal of $H_K$. Hence, there exists $x \in H_K^{*}$ such that $x^2=\varepsilon\alpha$ for some $\varepsilon \in \mathcal{O}_{H_K}^{*}$, and, since by construction $\varepsilon$ becomes a square in $L$, we conclude that $\alpha$ is a square in $L$.  
\end{proof} 
\end{proposition}

The following lemma is a trivial generalization of part of the proof of \cite[Lemma 1.1]{stoll}. We nevertheless include a proof for the sake of completeness.

\begin{lemma}\label{divisibility}
Let $K$ be a number field and $x^2+a\in K[x]$. Let $\{c_n\}_{n\geq 1}$ be the adjusted post-critical orbit of $f$ and fix $v\in \Sigma_K$. Then the following holds.
\begin{enumerate}
\item $v(c_n)<0$ if and only if $v(c_1)<0$, and if this holds then $v(c_n)=2^{n-1}v(c_1)$.
\item If $n=\min\{i\in \N\colon v(c_i)>0\}$, then $v(c_m)=v(c_n)$ if and only if $n\mid m$, and $v(c_m)=0$ otherwise.
\end{enumerate}
\end{lemma}
\begin{proof}
(1) is obvious. To prove (2), consider the finite prime $\p$ corresponding to $v$ and look at the post-critical orbit modulo powers of $\p$: one immediately sees that $0$ is periodic modulo $\p^{v(c_n)}$ and the period is $n$. On the other hand, modulo $\p^{v(c_n)+1}$ one has that $c_{n+1}=c_n^2+a\equiv a\equiv -c_1\bmod \p^{v(c_n)+1}$, and therefore $c_{n+2}\equiv c_2\bmod  \p^{v(c_n)+1}$. It follows that $v(c_i)\leq v(c_n)$ for every $i$.
\end{proof}
We are now in position to state and prove the second main theorem of this section. 
\begin{theorem} \label{evidence2}
Let $J\subseteq \bigoplus_{i\in \Z_{\geq 1}}\F_2$ be an $M$-coprime infinite set for some non-negative integer $M$. Let $K$ be a number field and $f=x^2+a\in K[x]$ with arboreal representation $\rho_f\coloneqq \rho_{f,0}$. Then we have the following.

\begin{enumerate}[(a)]
\item If $\im(\rho_f)\subseteq M_J$, then $f$ is post-critically finite and $\im(\rho_f)$ is finitely generated.
\item $\im(\rho_f)$ cannot coincide with $M_J$.
\end{enumerate}
\begin{proof}
$(a)$ Let $I$ be the infinite subset of the sub-vector space spanned by $J$ whose existence is guaranteed by the definition of $M$-coprimality. Since $\im(\rho_f)\subseteq M_J$, then $\im(\rho_f)\subseteq M_I$. Let $K'\coloneqq K(\sqrt{c_m}\colon m=1,\ldots,M)$. Now let $v \in I$ and let $i$ be the index guaranteed by the definition of $M$-coprimality. Notice that by definition of $M$-coprimality, up to replacing $I$ with an infinite subset of itself we can assume that $i>1$. By Proposition \ref{powerful_corollary}, we have that
\begin{equation}\label{squares}
\prod_{h \in \mathbb{Z}_{\geq 1}}c_h^{v_h} \in K^{* 2}.
\end{equation}

We claim that \eqref{squares} implies that $v(c_i)$ is even for all $v \in \Sigma_{K'}$. By Lemma \ref{divisibility}, this is certainly true if $v(c_i)<0$. Otherwise, observe that by definition of $K'$ we have that 
$$\prod_{h \in \mathbb{Z}_{\geq 1}: h>M}c_h^{v_h} \in {K'}^{* 2}.
$$
Lemma \ref{divisibility}, together with the definition of $M$-coprimality, ensures that for any $v \in \Sigma_{K'}$ we have that $v(c_i)>0$ implies that $v(c_h)=0$ for each $h>M$ such that $h \neq i$ and $v_h=1$. We therefore conclude that $v(c_i)$ must be even for each $v \in \Sigma_{K'}$. Now it follows from Proposition \ref{square_extension} that there is a finite extension $L/K'$ such that $c_i$ is a square in $L$. Since by definition of $M$-coprimality we have that $i$ is arbitrarily large, as $v$ varies in $I$, we conclude that if $f$ is post-critically infinite the curve $y^2=f^3(x)$ has infinitely many $L$-rational points. This contradicts Falting's theorem, and therefore $f$ must be post-critically finite. By Theorem \ref{PCF_equivalence}, this is equivalent to $\im(\rho_f)$ being finitely generated. 

Now we obtain immediately part $(b)$. Indeed thanks to part $(a)$ the image is topologically finitely generated as soon as it is included in $M_J$. But thanks to Lemma \ref{finite_generation}, it then follows that it cannot be equal to $M_J$ itself. 
\end{proof}
\end{theorem}
We now specialize Theorem \ref{evidence2} to a simple family of examples. Within this family one can find several subsets $J$ that are not $(k,\ell)$-progressing for any positive integers $k,\ell$. Hence this shows that Theorem \ref{evidence2} allows to exclude arboreal images (for polynomials of the shape $x^2+a$) also when Theorem \ref{evidence1} does not apply. 
\begin{corollary} \label{evidence2:conj}
Let $\{a_n\}_{n \in \mathbb{Z}_{\geq 1}}$ be a strictly increasing sequence of positive integers. Let $J\subseteq\bigoplus_{n\in\Z_{\geq 1}}\F_2$ be the set of vectors $\{v(n)\}_{n \in \mathbb{Z}_{\geq 1}}$ defined by $v(n)_i=1$ if and only if $i \leq a_n$. Then $M_J$ cannot be realized as $\im(\rho_{f,0})$ by any polynomial of the shape $x^2+a$ over any number field.
\begin{proof}
Thanks to Theorem \ref{evidence2} and Lemma \ref{finite_generation}, it is enough to show that $I$ is $M$-coprime for some $M \in \mathbb{Z}_{\geq 0}$. We claim that $J$ is $0$-coprime and that we can set $I$, in the definition of $0$-coprime, to be equal to $J$. For $n=1$, we put $i=1$ and we clearly have the required condition $(1)$ of $0$-coprimality. Now let $n>1$ and let $p_n$ be the largest prime that is at most $a_n$. Thanks to Bertrand's postulate, we must have that $2p_n>a_n$, and this readily guarantees condition $(1)$ of Definition \ref{M-coprimality definition} when we put $i=p_n$. Such a choice also guarantees condition $(2)$ of Definition \ref{M-coprimality definition}, since the set of prime numbers is infinite. It is therefore clear that $\{v(n)\}_{n\geq 1}$ satisfies the definition of $0$-coprimality, by choosing the $p_n$-th entry of $v(n)$ for every $n$. 
\end{proof}
\end{corollary}

\bibliographystyle{plain}
\bibliography{bibliography}

\end{document}